\documentclass[10pt]{amsart}

\usepackage{comment}

\usepackage{latexsym,amssymb}
\usepackage{amsmath,amsfonts,amsthm}
\input xypic

\newtheorem{theorem}{Theorem}[section]
\newtheorem{proposition}[theorem]{Proposition}
\newtheorem{corollary}[theorem]{Corollary}
\newtheorem{lemma}[theorem]{Lemma}

\newtheorem{remark}[theorem]{Remark}

\newcommand{\dom}{\mathbf{d}}
\newcommand{\ran}{\mathbf{r}}

\newcommand{\Hom}{\mathop{\mathrm{Hom}}\nolimits}
\newcommand{\Res}{\mathop{\mathrm{Res}}\nolimits}
\newcommand{\Ind}{\mathop{\mathrm{Ind}}\nolimits}

\begin{document}

\title[Inverse semigroups and groupoids]{The \'{e}tale groupoid of an inverse semigroup as a groupoid of filters}

\author{M.~V.~Lawson}
\address{Department of Mathematics\\
and the\\
Maxwell Institute for Mathematical Sciences\\
Heriot-Watt University\\
Riccarton\\
Edinburgh~EH14~4AS\\
Scotland}
\email{markl@ma.hw.ac.uk}

\author{S.~W.~Margolis}
\address{Department of Mathematics\\
Bar-Ilan University\\
52900 Ramat Gan\\
Israel}
\email{margolis@math.biu.ac.il}

\author{B.~Steinberg}
\address{School of Mathematics and Statistics\\
Carleton University\\
1125 Colonel By Drive\\
Ottawa\\
Ontario K1S 5B6\\
Canada}
\email{bsteinbg@math.carleton.ca}

\begin{abstract}
Paterson showed how to construct an \'etale groupoid from an inverse semigroup using ideas from functional analysis.
This construction was later simplified by Lenz.
We show that Lenz's construction can itself be further simplified by using filters:
the topological groupoid associated with an inverse semigroup is precisely a groupoid of filters.
In addition, idempotent filters are closed inverse subsemigroups and so determine transitive representations by means of partial bijections.
This connection between filters and representations by partial bijections is exploited
to show how linear representations of inverse semigroups can be constructed from the groups occuring in the associated topological groupoid.\\

\end{abstract}

\maketitle

\begin{center}
This paper is dedicated to the memory of our friend and colleague\\ Steve Haataja
\end{center}

\section{Introduction and motivation}\setcounter{theorem}{0}

In his influential book, Renault \cite{Renault} showed how to construct $C^{\ast}$-algebras from locally compact topological groupoids.
This can be seen as a far-reaching generalization of both commutative $C^{\ast}$-algebras and finite dimensional $C^{\ast}$-algebras.
From this perspective, locally compact topological groupoids can be viewed as `non-commutative topological spaces'.
Renault also showed that in addition to groupoids and $C^{\ast}$-algebras, a third class of structures naturally intervenes: inverse semigroups.
Local bisections of topological groupoids form inverse semigroups and, conversely, inverse semigroups can be used to construct topological groupoids.

The relationship between inverse semigroups and topological groupoids can be seen as a generalization of that between (pre)sheaves of groups
and their corresponding display spaces,
since an inverse semigroup with central idempotents is a presheaf of groups over its semilattice of idempotents.
This relationship has been investigated by a number of authors:
notably Paterson \cite{Paterson}, Kellendonk \cite{Kellendonk1,Kellendonk2,Kellendonk3,Kellendonk4} and Resende \cite{R}.
Our paper is related to Paterson's work but mediated through a more recent redaction due to Daniel Lenz \cite{Lenz}.

We prove two main results.
First, we show that Lenz's construction of the topological groupoid can be interpreted entirely in terms of down-directed cosets on inverse semigroups
---
these are precisely the filters in an inverse semigroup.
Such filters arise naturally from those transitive actions which we term `universal'.
Second, we show how representations of an inverse semigroup can be constructed from the groups occuring in the associated topological groupoid.
This is related to Steinberg's results on constructing finite-dimensional representations of inverse semigroups
using groupoid techniques described in \cite{S}.
The first result proved in this paper has already been developed further in \cite{Law4,LL}.

Lenz \cite{Lenz} was the main spur that led us to write this paper but in the course of doing so,
we realized that the first four chapters of Ruyle's unpublished thesis \cite{Ruyle}
could be viewed as a major contribution to the aims of this paper in the case of free inverse monoids.
Ruyle's work has proved indispensible for our Section~2.
In addition, Leech \cite{Leech}, with its emphasis on the order-theoretic structure of inverse semigroups,
can be seen with mathematical hindsight to be a precursor of our approach.
Last, but not least, Boris Schein in a number of seminars talked about ways of constructing {\em infinitesimal elements}
of an inverse semigroup: the maximal filters of an inverse semigroup can be regarded as just that \cite{Schein3,Schein4}.

For general inverse semigroup theory we refer the reader to \cite{Law2}.
However, we note the following.
The product in a semigroup will usually be denoted by concatenation
but sometimes we shall use $\cdot$ for emphasis; we shall also use it to denote actions.
In an inverse semigroup $S$ we define
$$\dom (s) = s^{-1}s \mbox{ and } \ran (s) = ss^{-1}.$$
Green's relation $\mathcal{H}$ can be defined in terms of this notation as follows:
$s \mathcal{H} t$ if and only if $\dom (s) = \dom (t)$ and $\ran (s) = \ran (t)$.
If $e$ is an idempotent in a semigroup $S$ then $G_{e}$ will denote the $\mathcal{H}$-class in $S$ containing $e$;
this is a maximal subgroup.
The natural partial order will be the only partial order considered when we deal with inverse semigroups.
If $X \subseteq S$ then $E(X)$ denotes the set of idempotents in $X$.
An inverse subsemigroup of $S$ is said to be {\em wide} if it contains all the idempotents of $S$.
A {\em primitive idempotent} $e$ in an inverse semigroup $S$ with zero is one with the property that
if $f \leq e$ then either $f = e$ or $f = 0$.
Let $S$ be an inverse semigroup.
The {\em minimum group congruence} $\sigma$ on $S$ is defined by $a \, \sigma \, b$ iff $c \leq a,b$ for some $c \in S$.
This congruence has the property that $S/\sigma$ is a group,
and if $\rho$ is any congruence on $S$ for which $S/\rho$ is a group,
we have that $\sigma \subseteq \rho$.
We denote by $\sigma^{\natural}$ the associated natural homomorphism $S \rightarrow S/\sigma$.
See \cite{Law2} for more information on this important congruence.

\section{The structure of transitive actions}\setcounter{theorem}{0}

In this section, we shall begin by reviewing the general theory of representations of inverse semigroups by partial permutations.
Chapter~IV, Section~4 of \cite{Petrich} contains an exposition of this elementary theory and we refer the reader there for any proofs we omit.
We also incorporate some results by Ruyle from \cite{Ruyle} which can be viewed as anticipating some of the ideas in this paper.
We then introduce the concept of universal transitive actions which provides the connection with the work of Lenz to be explained in Section~3.

\subsection{The classical theory}

A {\em representation} of an inverse semigroup by means of partial bijections (or partial permutations)
is a homomorphism $\theta \colon \: S \rightarrow I(X)$ to the symmetric inverse monoid on a set $X$.
A representation of an inverse semigroup in this sense leads to a corresponding notion of an action of the inverse semigroup $S$ on the set $X$:
the associated action is defined by $s \cdot x = \theta (s)(x)$, if $x$ belongs to the set-theoretic domain of $\theta (s)$.
The action is therefore a partial function from $S \times X$ to $X$ mapping $(s,x)$ to $s \cdot x$ when $\exists s \cdot x$
satisfying the two axioms:
\begin{description}
\item[{\rm (A1)}] If $\exists e \cdot x$ where $e$ is an idempotent then $e \cdot x = x$.

\item[{\rm (A2)}] $\exists (st) \cdot x$ iff $\exists s \cdot (t \cdot x)$ in which case they are equal.
\end{description}
It is easy to check that representations and actions are different ways of describing the same thing.
For convenience, we shall use the words `action' and `representation' interchangeably:
if we say the inverse semigroup $S$ acts on a set $X$ then this will imply the existence
of an appropriate homomorphism from $S$ to $I(X)$.
If $S$ acts on $X$ we shall often refer to $X$ as a {\em space} or as an {\em $S$-space} and its elements as {\em points}.
A subset $Y \subseteq X$ closed under the action is called a {\em subspace}.
Disjoint unions of actions are again actions.
An action is said to be {\em effective} if for each $x \in X$ there is $s \in S$ such that $\exists s \cdot x$.
We shall assume that all our actions are effective.
An effective action of an inverse semigroup $S$ on the set $X$
induces an equivalence relation $\sim$ on the set $X$ when we define
$x \sim y$ iff $s \cdot x = y$ for some $s \in S$.
The action is said to be {\em transitive} if $\sim$ is $X \times X$.
Just as in the theory of permutation representations of groups,
every representation of an inverse semigroup is a disjoint union of transitive representations.
Thus the transitive representations of inverse semigroups are of especial significance.

Let $X$ and $Y$ be $S$-spaces.
A {\em morphism} from $X$ to $Y$ is a function $\alpha \colon \: X \rightarrow Y$
such that $\exists s \cdot x$ implies that $\exists s \cdot \alpha (x)$ and $\alpha (s \cdot x) = s \cdot \alpha (x)$.
A {\em strong morphism} from $X$ to $Y$ is a function $\alpha \colon \: X \rightarrow Y$
such that $\exists s \cdot x \Leftrightarrow \exists s \cdot \alpha (x)$ and if $\exists s \cdot x$ then $\alpha (s \cdot x) = s \cdot \alpha (x)$.
Bijective strong morphisms are called {\em equivalences}.
The proofs of the following two lemmas are straightforward.

\begin{lemma}\mbox{}
\begin{description}

\item[{\rm (i)}] Identity functions are (strong) morphisms.

\item[{\rm (ii)}] The composition of (strong) morphisms is again a (strong) morphism.

\end{description}
\end{lemma}

\begin{lemma} Let $S$ be an inverse semigroup acting on $X$, $Y$ and $Z$
\begin{description}

\item[{\rm (i)}] The image of a strong morphism $\alpha \colon \: X \rightarrow Y$  is a subspace of $Y$.

\item[{\rm (ii)}] If $X$ and $Y$ are transitive $S$-spaces and $\alpha \colon \: X \rightarrow Y$ is a strong morphism
then $\alpha$ is surjective.

\end{description}
\end{lemma}

If we fix an inverse semigroup $S$ there are a number of categories of actions associated with it:
actions and morphisms, actions and strong morphisms,
transitive actions and morphisms, and transitive actions and strong morphisms.
As we indicated above, these two categories of transitive actions will be of central importance.

A {\em congruence} on $X$ is an equivalence relation $\sim$ on the set $X$ such that
if $x \sim y$ and if $\exists s \cdot x$ and $\exists s \cdot y$
then $s \cdot x \sim s \cdot y$.
A {\em strong congruence} on $X$ is an equivalence relation $\approx$ on the set $X$ such that
if $x \approx y$ and $s \in S$  we have that $\exists s \cdot x \Leftrightarrow \exists s \cdot y$,
and if the actions are defined then $s \cdot x \approx s \cdot y$.

Strong morphisms and strong congruences are united by a classical first isomorphism theorem.
Recall that the {\em kernel} of a function is the equivalence relation induced on its domain.
The proofs of the following are routine.

\begin{proposition} \mbox{}
\begin{description}

\item[{\rm (i)}] Let $\alpha \colon \: X \rightarrow Y$ be a strong morphism.
Then the kernel of $\alpha$ is a strong congruence.

\item[{\rm (ii)}] Let $\sim$ be a strong congruence on $X$.
Denote the $\sim$-class containing the element $x$ by $[x]$.
Define $s \cdot [x] = [s \cdot x]$ if $\exists s \cdot x$.
Then this defines an action $S$ on the set of $\sim$-congruence classes $X/\sim$
and the natural map $\nu \colon \: X \rightarrow X/\sim$ is a strong morphism.

\item[{\rm (iii)}] Let $\alpha \colon \: X \rightarrow Y$ be a strong morphism,
let its kernel be $\sim$ and let $\nu \colon \: X \rightarrow X/\sim$ be the associated natural map.
Then there is a unique injective strong morphism
$\beta \colon \: X/\sim \rightarrow Y$ such that $\beta \nu = \alpha$.

\end{description}
\end{proposition}

The above result tells us that the category of transitive representations of a fixed inverse semigroup
with strong morphisms between them has a particularly nice structure.

We may analyze transitive actions of inverse semigroups in a way generalizing the relationship between transitive group actions and subgroups.
To describe this relationship we need some definitions.
If $A \subseteq S$ is a subset then define
$$A^{\uparrow} = \{s \in S \colon \: a \leq s \mbox{ for some } a \in A \}.$$
If $A = A^{\uparrow}$ then $A$ is said to be {\em closed (upwards)}.

Let $X$ be an $S$-space.
Fix a point $x \in X$, and consider the set $S_{x}$ consisting of all
$s \in S$ such that $s \cdot x = x$.
We call $S_{x}$ the {\em stabilizer} of the point $x$.

\begin{remark}
{\em We do not assume in this paper that homomorphisms of inverse semigroups with zero preserve the zero.
If $\theta \colon S \rightarrow I(X)$ is a representation that does preserve zero then the zero of $S$ is mapped to the empty function of $I(X)$.
Clearly, the empty function cannot belong to any stabilizer.
We say that a closed inverse subsemigroup is {\em proper} if it does not contain a zero.
In the theory we summarize below, proper closed inverse subsemigroups
arise from actions where the zero acts as the empty partial function.}
\end{remark}

Now let $y \in X$ be any point.
By transitivity, there is an element $s \in S$ such that $s \cdot x = y$.
Observe that because $s \cdot x$ is defined so too is $s^{-1}s$ and that $s^{-1}s \in S_{x}$.
The set of all elements of $S$ which map $x$ to $y$ is $(sS_{x})^{\uparrow}$.

Let $H$ be a closed inverse subsemigroup of $S$.
Define a {\em left coset} of $H$ to be a set of the form $(sH)^{\uparrow}$ where $s^{-1}s \in H$.
We give the proof of the following for completeness.

\begin{lemma} \mbox{}
\begin{description}

\item[{\rm (i)}] Two cosets $(sH)^{\uparrow}$ and $(tH)^{\uparrow}$ are equal iff $s^{-1}t \in H$.

\item[{\rm (ii)}] If $(sH)^{\uparrow} \cap (tH)^{\uparrow} \neq \emptyset$ then $(sH)^{\uparrow} = (tH)^{\uparrow}$.

\end{description}
\end{lemma}
\proof (i) Suppose that  $(sH)^{\uparrow} = (tH)^{\uparrow}$.
Then $t \in (sH)^{\uparrow}$ and so $sh \leq t$ for some $h \in H$.
Thus $s^{-1}sh \leq s^{-1}t$.
But $s^{-1}sh \in H$ and $H$ is closed and so $s^{-1}t \in H$.

Conversely, suppose that $s^{-1}t \in H$.
Then $s^{-1}t = h$ for some $h \in H$ and so
$sh = ss^{-1}t \leq t$.
It follows that $tH \subseteq sH$ and so $(tH)^{\uparrow} \subseteq (sH)^{\uparrow}$.
The reverse inclusion follows from the fact that $t^{-1}s \in H$ since
$H$ is closed under inverses.

(ii) Suppose that $a \in (sH)^{\uparrow} \cap (tH)^{\uparrow}$.
Then $sh_{1} \leq a$ and $th_{2} \leq a$ for some $h_{1},h_{2} \in H$.
Thus $s^{-1}sh_{1} \leq s^{-1}a$ and $t^{-1}th_{2} \leq t^{-1}a$.
Hence $s^{-1}a,t^{-1}a \in H$.
It follows that $s^{-1}aa^{-1}t \in H$,
but $s^{-1}aa^{-1}t \leq s^{-1}t$.
This gives the result by (i) above.\qed\\

We denote by $S/H$ the set of all left cosets of $H$ in $S$.
The inverse semigroup $S$ acts on the set $S/H$ when we define
$$a \cdot (sH)^{\uparrow} = (asH)^{\uparrow} \mbox{ whenever} \dom (as) \in H.$$
This defines a transitive action.
The following is Lemma~IV.4.9 of \cite{Petrich} and Proposition~5.8.5 of \cite{H}.

\begin{theorem} Let $S$ act transitively on the set $X$.
Then the action is equivalent to the action of $S$ on the set $S/S_{x}$ where $x$ is any point of $X$.
\end{theorem}

The following is Proposition~IV.4.13 of \cite{Petrich}.

\begin{proposition}
If $H$ and $K$ are any closed inverse subsemigroups of $S$ then they determine
equivalent actions if and only if there exists $s \in S$ such that
$$sHs^{-1} \subseteq K \mbox{ and } s^{-1}Ks \subseteq H.$$
\end{proposition}

The above relationship between closed inverse subsemigroups is called {\em conjugacy}
and defines an equivalence relation on the set of closed inverse subsemigroups.
The proof of the following is given for completeness.

\begin{lemma}
$H$ and $K$ are conjugate if and only if
$$(sHs^{-1})^{\uparrow} = K
\mbox{ and }
(s^{-1}Ks)^{\uparrow} = H.$$
\end{lemma}
\proof Let $H$ and $K$ be conjugate.
Let $e \in H$ be any idempotent.
Then $ses^{-1} \in K$.
But $ses^{-1} \leq ss^{-1}$ and so $ss^{-1} \in K$.
Similarly $s^{-1}s \in H$.
We have that $sHs^{-1} \subseteq K$ and so $(sHs^{-1})^{\uparrow} \subseteq K$.
Let $k \in K$.
Then $s^{-1}ks \in H$ and $s(s^{-1}ks)s^{-1} \in sHs^{-1}$ and
$s(s^{-1}ks)s^{-1} \leq k$.
Thus  $(sHs^{-1})^{\uparrow} = K$, as required.
The converse is immediate.\qed\\

Thus to study the transitive actions of an inverse semigroups $S$ it is enough to study the closed inverse subsemigroups of $S$ up to conjugacy.

The following result is motivated by Lemma~2.16 of Ruyle's thesis \cite{Ruyle} and brings morphisms and strong morphisms back into the picture.

\begin{theorem} Let $S$ be an inverse semigroup acting transitively on the sets $X$ and $Y$, and let $x \in X$ and $y \in Y$.
Let $S_{x}$ and $S_{y}$ be the stabilizers in $S$ of $x$ and $y$ respectively.
\begin{description}

\item[{\rm (i)}] There is a (unique) morphism $\alpha \colon \: X \rightarrow Y$ such that $\alpha (x) = y$
iff $S_{x} \subseteq S_{y}$.

\item[{\rm (ii)}] There is a (unique) strong morphism $\alpha \colon \: X \rightarrow Y$ such that $\alpha (x) = y$
iff $S_{x} \subseteq S_{y}$ {\em and } $E(S_{x}) = E(S_{y})$.

\end{description}
\end{theorem}
\proof (i) We begin by proving uniqueness.
Let $\alpha,\beta \colon \: X \rightarrow Y$ be morphisms such that $\alpha (x) = \beta (x) = y$.
Let $x' \in X$ be arbitrary.
By transitivity there exists $a \in S$ such that $x' = a \cdot x$.
By the definition of morphisms we have that
$\exists a \cdot \alpha (x)$ and $\exists a \cdot \beta (x)$
and that
$$\alpha (x') = \alpha (a \cdot x) = a \cdot \alpha (x)$$
and
$$\beta (x') = \beta (a \cdot x) = a \cdot \beta (x).$$
But by assumption $\alpha (x) = \beta (x) = y$ and so $\alpha (x') = \beta (x')$.
It follows that $\alpha = \beta$.

Let $\alpha \colon \: X \rightarrow Y$ be a morphism such that $\alpha (x) = y$.
Let $s \in S_{x}$.
Then $\exists s \cdot x$ and $s \cdot x = x$.
By the definition of morphism, it follows that $\exists s \cdot \alpha (x)$
and that $\alpha (s \cdot x) = s \cdot \alpha (x)$.
But $s \cdot x = x$ and so $\alpha (x) = s \cdot \alpha (x)$.
Hence $s \cdot y = y$.
We have therefore proved that $s \in S_{y}$,
and so $S_{x} \subseteq S_{y}$.

Suppose now that $S_{x} \subseteq S_{y}$.
We have to define a morphism $\alpha \colon \:X \rightarrow Y$ such that $\alpha (x) = y$.
We start by defining $\alpha (x) = y$.
Let $x' \in X$ be any point in $X$.
Then $x' = a \cdot x$ for some $a \in S$.
We need to show that $a \cdot y$ exists.
Since $a \cdot x$ exists we know that $a^{-1}a \cdot x$ exists and this is equal to $x$.
It follows that $a^{-1}a \in S_{x}$ and so $a^{-1}a \in S_{y}$, by assumption.
Thus $a^{-1}a \cdot y$ exists and is equal to $y$.
But from the existence of $a^{-1}a \cdot y$ we can deduce the existence of $a \cdot y$.
We would therefore like to define $\alpha (x') = a \cdot y$.
We have to check that this is well-defined.
Suppose that $x' = a \cdot x = b \cdot x$.
Then $b^{-1}a \cdot x = x$ and so $b^{-1}a \in S_{x}$.
By assumption, $b^{-1}a \in S_{y}$ and so
$b^{-1}a \cdot y = y$.
Thus $bb^{-1}a \cdot y = b \cdot y$
and $bb^{-1}a \cdot y = bb^{-1} \cdot (a \cdot y) = a \cdot y$.
Thus $a \cdot y = b \cdot y$.
It follows that $\alpha$ is a well-defined function mapping $x$ to $y$.
It remains to show that $\alpha$ is a morphism.
Suppose that $s \cdot x'$ is defined.
By assumption, there exists $a \in S$ such that $x' = a \cdot x$.
By definition $\alpha (x') = a \cdot y$.
We have that $ s \cdot x' = s \cdot (a \cdot x) = sa \cdot x$.
By definition $\alpha (s \cdot x') = sa \cdot y$.
But $sa \cdot y = s \cdot (a \cdot y) = s \cdot \alpha (x')$.
Hence $\alpha (s \cdot x') = s \cdot \alpha (x')$, as required.

(ii) We begin by proving uniqueness.
Let $\alpha,\beta \colon \: X \rightarrow Y$ be strong morphisms such that $\alpha (x) = \beta (x) = y$.
Let $x' \in X$ be arbitrary.
By transitivity there exists $a \in S$ such that $x' = a \cdot x$.
By the definition of strong morphisms we have that
$\exists a \cdot \alpha (x)$ and $\exists a \cdot \beta (x)$
and that
$$\alpha (x') = \alpha (a \cdot x) = a \cdot \alpha (x)$$
and
$$\beta (x') = \beta (a \cdot x) = a \cdot \beta (x).$$
But by assumption $\alpha (x) = \beta (x) = y$ and so $\alpha (x') = \beta (x')$.
It follows that $\alpha = \beta$.

Next we prove existence.
Suppose that $S_{x} \subseteq S_{y}$ and $E(S_{x}) = E(S_{y})$.
We have to define a strong morphism $\alpha \colon \:X \rightarrow Y$ such that $\alpha (x) = y$.
We start by defining $\alpha (x) = y$.
Let $x' \in X$ be any point in $X$.
Then $x' = a \cdot x$ for some $a \in S$.
We need to show that $a \cdot y$ exists.
Since $a \cdot x$ exists we know that $a^{-1}a \cdot x$ exists and this is equal to $x$.
It follows that $a^{-1}a \in S_{x}$ and so $a^{-1}a \in S_{y}$, by assumption.
Thus $a^{-1}a \cdot y$ exists and is equal to $y$.
But from the existence of $a^{-1}a \cdot y$ we can deduce the existence of $a \cdot y$.
We therefore define $\alpha (x') = a \cdot y$.
We have to check that this is well-defined.
Suppose that $x' = a \cdot x = b \cdot x$.
Then $b^{-1}a \cdot x = x$ and so $b^{-1}a \in S_{x}$.
By assumption, $b^{-1}a \in S_{y}$ and so
$b^{-1}a \cdot y = y$.
Thus $bb^{-1}a \cdot y = b \cdot y$
and $bb^{-1}a \cdot y = bb^{-1} \cdot (a \cdot y) = a \cdot y$.
Thus $a \cdot y = b \cdot y$.
It follows that $\alpha$ is a well-defined function mapping $x$ to $y$.

It remains to show that $\alpha$ is a strong morphism.
Suppose that $s \cdot x'$ is defined.
By assumption, there exists $a \in S$ such that $x' = a \cdot x$.
By definition $\alpha (x') = a \cdot y$.
We have that $ s \cdot x' = s \cdot (a \cdot x) = sa \cdot x$.
By definition $\alpha (s \cdot x') = sa \cdot y$.
But $sa \cdot y = s \cdot (a \cdot y) = s \cdot \alpha (x')$.
Hence $\alpha (s \cdot x') = s \cdot \alpha (x')$.

Now suppose that $\alpha (x') = y'$ and $\exists s \cdot y'$.
We shall prove that $\exists s \cdot x'$.
Observe that $\exists s^{-1}s \cdot y'$ and that it is enough to prove that $\exists s^{-1}s \cdot x'$.
Let $x' = u \cdot x$, which exists since we are assuming that our action is transitive.
Then by what we proved above we have that $y' = u \cdot y$.
Observe that $u^{-1}(s^{-1}s)u \cdot y = y$ and so $u^{-1}(s^{-1}s)u \in E(S_{y})$.
It follows by our assumption that $u^{-1}(s^{-1}s)u \in E(S_{x})$
and so  $u^{-1}(s^{-1}s)u \cdot x = x$.
It readily follows that $\exists s^{-1}s \cdot x'$,
and so $\exists s \cdot x'$, as required.

We now prove the converse.
Let $\alpha \colon \: X \rightarrow Y$ be a strong morphism such that $\alpha (x) = y$.
Let $s \in S_{x}$.
Then $\exists s \cdot x$ and $s \cdot x = x$.
By the definition of strong morphism, it follows that $\exists s \cdot \alpha (x)$
and that $\alpha (s \cdot x) = s \cdot \alpha (x)$.
But $s \cdot x = x$ and so $\alpha (x) = s \cdot \alpha (x)$.
Hence $s \cdot y = y$.
We have therefore proved that $s \in S_{y}$, and so $S_{x} \subseteq S_{y}$.
Let $e \in E(S_{y})$.
Then $\exists e \cdot \alpha (x)$.
But $\alpha$ is a strong morphism and so $\exists e \cdot x$.
Clearly $e \in E(S_{x})$.
It follows that $E(S_{x}) = E(S_{y})$.\qed\\

The following result is adapted from Lemma~1.9 of Ruyle \cite{Ruyle} and will be useful to us later.

\begin{lemma} Let $F$ be a closed inverse subsemigroup of the semilattice of idempotents of the inverse subsemigroup $S$.
Define
$$\overline{F} = \{s \in S \colon \: s^{-1}Fs \subseteq F, sFs^{-1} \subseteq F \}.$$
Then $\overline{F}$ is a closed inverse subsemigroup of $S$ whose semilattice of idempotents is $F$.
Furthermore, if $T$ is any closed subsemigroup of $S$ with semilattice of idempotents $F$ then $T \subseteq \overline{F}$.
\end{lemma}
\proof
Clearly the set $\overline{F}$ is closed under inverses.
Let $s,t \in \overline{F}$.
We calculate
$$(st)^{-1}F(st) = t^{-1}(s^{-1}Fs)t \subseteq t^{-1}Ft \subseteq F$$
and
$$(st)F(st)^{-1} = s(tFt^{-1})s^{-1} \subseteq sFs^{-1} \subseteq F.$$
Thus $st \in \overline{F}$.
It follows that $\overline{F}$ is an inverse subsemigroup of $S$.

Let $e \in \overline{F}$
and
$f \in F$.
Then by assumption $ef \in F$.
But $ef \leq e$ and $F$ is a closed inverse subsemigroup of the semilattice of idempotents
and so $e \in F$.
Thus $E(\overline{F}) = F$.

Let $s \leq t$ where $s \in \overline{F}$.
Then $s = ss^{-1}t = ft$.
Let $e \in F$.
Then
$$s^{-1}es = t^{-1}feft = t^{-1}eft \leq t^{-1}et.$$
Now $s^{-1}es,t^{-1}et$ are idempotents and $s^{-1}es \in F$
thus $t^{-1}et \in F$, because $F$ is a closed inverse subsemigroup of the semilattice of idempotents.
Similarly $tet^{-1} \in F$.
It follows that $t \in \overline{F}$ and so $\overline{F}$ is a closed inverse subsemigroup of $S$.

Finally, let $T$ be a closed inverse subsemigroup of $S$ such that $E(T) = F$.
Let $t \in T$.
Then for each $e \in F$ we have that
$t^{-1}et,tet^{-1} \in F$.
Thus $T \subseteq \overline{F}$.\qed\\

A closed inverse subsemigroup $T$ of $S$ will be said to be {\em fully closed} if $T = \overline{E(T)}$.
Closed inverse subsemigroups of the semilattice of idempotents of an inverse semigroup are called filters {\em in} $E(S)$.
Observe the emphasis on the word `in'.
A filter in $E(S)$ is said to be {\em principal} if it is of the form $e^{\uparrow}$.
We denote by $\mathcal{F}_{E(S)}$ the set of all closed inverse subsemigroups of $E(S)$
and call it the {\em filter space of the semilattice of idempotents of $S$}.
This filter space is a poset when we define $F \leq F'$ iff $F' \subseteq F$
so that, in particular, $e^{\uparrow} \leq f^{\uparrow}$ iff $e \leq f$.

Let $F$ be a filter in $E(S)$.
Then $F^{\uparrow}$ is a closed inverse subsemigroup containing $F$ and clearly the smallest such inverse subsemigroup.
On the other hand, by Lemma~2.10, $\overline{F}$ is the largest closed inverse subsemigroup
with semilattice of idempotents $F$.
We have therefore proved the following.

\begin{lemma} The semilattice of idempotents of any closed inverse subsemigroup $H$ of an inverse semigroup $S$ is a filter $F$ in $E(S)$
and $F^{\uparrow} \subseteq H \subseteq \overline{F}$.
Thus $F^{\uparrow}$ is the smallest closed inverse subsemigroup with semilattice of idempotents $F$ and $\overline{F}$ is the largest.
\end{lemma}

\begin{proposition} Let $S$ be an inverse semigroup and let $G = S/\sigma$.
Then there is an inclusion-preserving bijection between the wide closed inverse subsemigroups of $S$ and the subgroups of $G$.
\end{proposition}
\proof Let $E(S) \subseteq T \subseteq S$ be a wide inverse subsemigroup.
Then the image of $T$ in $G$ is a subgroup since inverse subsemigroups map to inverse subsemigroups under homomorphisms.
Suppose $T$ and $T'$, where also $E(S) \subseteq T' \subseteq S$, have the same image in $G$.
Let $t \in T$.
Then $\sigma^{\natural}(t) = \sigma^{\natural}(t')$ for some $t' \in T'$.
Thus $a \leq t,t'$ from the definition of $\sigma$.
But both $T$ and $T'$ are order ideals of $S$ and so $a \in T \cap T'$.
Thus $a \leq t$ and $a \in T'$ and $T'$ is closed thus $t \in T'$.
We have shown that $T \subseteq T'$.
The reverse inclusion follows by symmetry.
If $H$ is a subgroup of $G$ then the full inverse image of $H$ under $\sigma^{\natural}$
is a wide inverse subsemigroup of $S$.
This defines an order-preserving map going in the opposite direction.
It is now clear that the result holds.
\qed\\

The following is a special case of Lemma~2.17 of \cite{Ruyle}.
We include it for interest since we shall not use it explicitly.

\begin{lemma} Let $F$ be a filter in $E(S)$ in the inverse semigroup $S$.
\begin{description}

\item[{\rm (i)}] The intersection of any family of closed inverse subsemigroups with common semilattice of idempotents $F$ is again
a closed inverse subsemigroup with semilattice of idempotents $F$.

\item[{\rm (ii)}] Given any family of closed inverse subsemigroups with common semilattice of idempotents $F$ there is a smallest closed
inverse subsemigroup with semilattice $F$ which contains them all.

\end{description}
\end{lemma}

\subsection{Universal and fundamental transitive actions}

We shall now define two special classes of transitive actions that play a decisive role in this paper.
Let $S$ be an inverse semigroup and let $H$ be a closed inverse subsemigroup of $S$.
By Lemma~2.10, we have that
$$E(H)^{\uparrow} \subseteq H \subseteq \overline{E(H)}$$
where $E(H)$ is a filter in $E(S)$.
We shall use this observation as the basis of two definitions,
the first of which is by far the most important.
We shall say that a transitive $S$-space $X$ is {\em universal} if the stabilizer of a point of $X$ is the closure $F^{\uparrow}$ for some filter $F$ of $E(S)$,
and {\em fundamental} if the stabilizer of a point of $X$ is $\overline{F}$ for some filter $F$ in $E(S)$.
Both definitions are independent of the point chosen.

\begin{lemma} \mbox{}
\begin{enumerate}

\item A strong morphism between universal transitive actions is an equivalence.

\item Any strong morphism with domain a fundamental transitive action and codomain a transitive action is an equivalence.

\end{enumerate}
\end{lemma}
\proof (1) Let $X$ and $Y$ be universal transitive spaces.
Let $\alpha \colon \: X \rightarrow Y$ be a strong morphism.
Choose $x \in X$.
Then $S_{x} \subseteq S_{\alpha (x)}$ and $E(S_{x}) = E(S_{\alpha (x)})$.
But the actions are universal and so all stabilizers are the full closures of their semilattices of idempotents.
Thus $S_{x} = S_{\alpha (x)}$ and so $\alpha$ is an equivalence by Theorem~2.9(ii).

(2)  Let $X$ and $Y$ be transitive spaces where $X$ is fundamental and let $\alpha \colon \: X \rightarrow Y$ be a strong morphism.
Choose $x \in X$ and let $y = \alpha (x)$.
Then $S_{x} \subseteq S_{y}$ and $E(S_{x}) = E(S_{y})$ by Theorem~2.9(ii).
But $S_{x}$ is fundamental and so $S_{x} = S_{y}$.
We may deduce from Theorem~2.9(ii) that there is a unique strong morphism from $Y$ to $X$ mapping $y$ to $x$.
It follows that $\alpha$ is an equivalence.\qed \\

If $\alpha \colon X \rightarrow Y$ is a strong morphism between two transitive $S$-spaces,
we shall say that $Y$ is {\em strongly covered} by $X$.
The importance of universal actions arises from the following result.

\begin{proposition} Let $S$ be an inverse semigroup.
\begin{enumerate}

\item Each transitive action of $S$ is strongly covered by a universal one.

\item Each transitive action of $S$ strongly covers a fundamental one.

\end{enumerate}
\end{proposition}
\proof (1) Let $Y$ be an arbitrary transitive $S$-space.
Choose a point $y \in Y$.
Let $F = E(S_{y})$ and put $H = F^{\uparrow}$.
Then $E(H) = E(S_{y})$ and $H \subseteq S_{y}$.
Put $X = S/H$ and choose the point $x$ in $X$ to be the coset $H$.
Then there is a unique strong morphism $\alpha \colon \: X \rightarrow Y$ such that $\alpha (x) = y$
by Theorem~2.9(ii) which is surjective by Lemma~2.2(ii) and $X$ is a universal transitive space by construction.

(2) Let $Y$ be an arbitrary transitive $S$-space.
Choose a point $y \in Y$.
Let $F = E(S_{y})$ and put $H = \overline{F}$.
Thus by Lemma~2.10 we have that $S_{y} \subseteq H$ and $E(S_{y}) = E(H)$.
Put $X = S/H$ and choose the point $x$ in $X$ to be the coset $H$.
Then there is a unique strong morphism $\alpha \colon \: Y \rightarrow X$ such that $\alpha (y) = x$ by Theorem~2.9(ii)
which is surjective by Lemma~2.2(ii) and $X$ is a fundamental transitive space by construction.\qed \\

\begin{theorem} Let $X$ be a universal, transitive $S$-space and let $x$ be a point of $X$.
Put $S_{x} = F^{\uparrow}$, where $F$ is a filter in $E(S)$ and $G_{F} = \overline{F}/\sigma$.
Then there is an order-preserving bijection between the set of strong congruences on $X$ and the set of subgroups of $G_{F}$.
\end{theorem}
\proof Put $G = G_{F}$.
By Proposition~2.12,
there is an order-preserving bijection between the closed inverse subsemigroups $H$ such that
$F^{\uparrow} \subseteq H \subseteq \overline{F}$ and the subgroups of $G$.
Thus we need to show that there is a bijection between the set of strong congruences on $X$
and the set of closed wide inverse subsemigroups of $\overline{F}$.
Observe that we use the fact that strong morphisms between transitive spaces are surjective by Lemma~2.2(ii).

Let $\sim$ be a strong congruence defined on $X$.
Then by Proposition~2.3 it determines a strong morphism $\nu \colon \: X \rightarrow X/\sim$.
For $x$ given in the statement of the theorem,
we have that the stabilizer of $[x]$, the $\sim$-class containing $x$,
is a closed inverse subsemigroup $H_{x}$ such that $F^{\uparrow} \subseteq H_{x} \subseteq \overline{F}$ by Theorem~2.9(ii).
We have thus defined a function from strong congruences on $X$ to the set of closed wide inverse subsemigroups of $\overline{F}$.

Suppose that $\sim_{1}$ and $\sim_{2}$ are two strong congruences on $X$ that map to the same closed wide inverse subsemigroup.
Denote the $\sim_{i}$ equivalence class containing $x$ by $[x]_{i}$ and let $\nu_{i} \colon X \rightarrow X/\sim_{i}$ be the natural map.
Let $x \in X$.
Then the stabilizer of $[x]_{1}$ and the stabilizer of $[x]_{2}$ are the same: namely $H$.
Suppose that $x \sim_{1} y$.
Thus $[x]_{1} = [y]_{1}$.
Since $X$ is an universal transitive $S$-space there is $b \in B$ such that $b \cdot x = y$.
It follows that $b \cdot [x]_{1} = [y]_{1} = [x]_{1}$ and so $b \in H$.
By assumption $b \cdot [x]_{2} = [x]_{2}$.
But $\sim_{2}$ is a strong congruence and so $y = b \cdot x \sim_{2} x$ and so $x \sim_{2} y$.
A symmetrical argument shows that $\sim_{1}$ and $\sim_{2}$ are equal.
Thus the correspondence we have defined is injective.
We now show that it is surjective.

Let $F^{\uparrow} \subseteq H \subseteq \overline{F}$ be such a closed wide inverse subsemigroup.
Then $Y = H/S$ is a transitive $S$-space.
Choose the point $y = H \in Y$.
Then by Theorem~2.9(ii) there is a unique strong morphism $\alpha_{H} \colon \: X \rightarrow Y$ such that $\alpha (x) = y$.
The kernel of $\alpha_{H}$, which we denote by $\sim_{H}$, is a strong congruence defined on $X$ by Proposition~2.3,
and the kernel of $\alpha_{H}$ maps to $H$.\qed \\

Observe that the above theorem requires a chosen point in $X$.

\subsection{A topological interpretation}

Let $S$ be an inverse semigroup and $X$ an $S$-space.
Define an $S$-labeled graph $G(X)$ whose vertices are X and whose edges go from $x$ to $sx$,
where $x \in X, s \in S$ and $sx$ is defined,
with label $s$ on this edge in this case.
There is an obvious involution on the graph by inversion, so this is a graph {\em in the sense of Serre}.
Observe that the directed graph $G(X)$ is connected iff $X$ is transitive.
From now on we shall deal only with transtive actions and so our graphs will be connected.

The {\em star} of a vertex $x$ in $G(X)$ is the set of all edges that start at $x$.
Now let $G$ and $H$ be arbitrary graphs.
A morphism $f$ from $G$ to $H$ is called an {\em immersion}
if it induces an injection from the star set of $x$ to that of $f(x)$ for each vertex $x$ of $G$.
The morphism $f$ is called a {\em cover} if it induces a bijection between such star sets.
The following is the key link between the algebraic and the topological interpretations of inverse semigroup actions.

\begin{lemma} Let $S$ be an inverse semigroup and let $X$ and $Y$ be transitive $S$-spaces.
There is a morphism from $X$ to $Y$ iff there is a label preserving immersion from $G(X)$ to $G(Y)$,
and there is a strong morphism from $X$ to $Y$ iff there is a label preserving cover from $G(X)$ to $G(Y)$.
\end{lemma}
\proof Let $\alpha \colon X \rightarrow Y$ be a morphism of transitive $S$-spaces.
Consider the directed edge $x \stackrel{s}{\rightarrow} y$ in the graph $G(X)$.
Then $s \cdot x = y$.
Since $\alpha$ is a morphism, we have that $\alpha (s \cdot x) = s \cdot \alpha (x) = \alpha (y)$.
We may therefore define $f \colon G(X) \rightarrow G(Y)$ by mapping the edge
$x \stackrel{s}{\rightarrow} y$ to the edge $\alpha (x) \stackrel{s}{\rightarrow} \alpha (y)$.
It is immediate that this is an immersion.
The fact that immersions arise from morphisms is now straightforward to prove.
Finally, suppose that $\alpha$ is a strong morphism.
Let $\alpha (x) \stackrel{s}{\rightarrow} \alpha (y)$ be an edge.
This means that $s \cdot \alpha (x) = \alpha (y)$.
But $\alpha$ is a strong morphism and so $s \cdot x$ exists and $\alpha (s \cdot x) = s \cdot \alpha (x)$.
It follows that the graph map is a cover.\qed \\

For a more complete account of the connection between immersions, inverse monoids and inverse categories see \cite{Margmeak,S1}.


\section{The \'etale groupoid associated with an inverse semigroup}\setcounter{theorem}{0}

In Section~2, we investigated the relationship between transitive actions of an inverse semigroup and closed inverse subsemigroups.
We found that the universal transitive actions played a special role.
We shall show in this section how these universal transitive actions, via their stabilizers,
lead to the inverse semigroup introduced by Lenz and thence to Paterson's \'etale groupoid.

\subsection{The inverse semigroup of cosets $\mathcal{K}(S)$}

We begin by reviewing a construction studied by a number of authors \cite{Schein2, Leech, Law1, Law2}.
A subset $A \subseteq S$ of an inverse semigroup is called an {\em atlas} if $A = AA^{-1}A$.
A closed atlas is precisely a coset of a closed inverse subsemigroup of $S$ \cite{Law1}.
We shall therefore refer to a closed atlas as a {\em coset}.
Observe that the intersection of cosets, if non-empty, is a coset.
The set of cosets of $S$ is denoted by $\mathcal{K}(S)$.
There is a product on $\mathcal{K}(S)$, denoted by $\otimes$,
and defined as follows: if $A,B \in \mathcal{K}(S)$ then $A \otimes B$ is the intersection of all cosets of $S$ that contain the set $AB$.
More explicitly if $X = (aH)^{\uparrow}$, where $a^{-1}a \in H$, and $Y = (bK)^{\uparrow}$, where $b^{-1}b \in K$,
then $X \otimes Y = (ab \langle b^{-1}Hb, K \rangle)^{\uparrow}$ where $\langle C, D \rangle$ is the inverse subsemigroup of $S$ generated by $C \cup D$.
In fact, $\mathcal{K}(S)$ is an inverse semigroup called the {\em (full) coset semigroup of} $S$.
Note that its natural partial order is reverse inclusion.
Thus $S$ is the zero element of $\mathcal{K}(S)$.
The idempotents of $\mathcal{K}(S)$ are just the closed inverse subsemigroups of $S$.

There is an embedding $\iota \colon \: S \rightarrow \mathcal{K}(S)$ that maps $s$ to $s^{\uparrow}$.
Observe now that if $A \in \mathcal{K}(S)$ then for each $s \in A$ we have that $s^{\uparrow} \subseteq A$ and so $A \leq s^{\uparrow}$.
It follows readily from this that $A$ is in fact the meet of the set $\{s^{\uparrow} \colon \: s \in A \}$.
More generally, every non-empty subset of $\mathcal{K}(S)$ has a meet and so the inverse semigroup $\mathcal{K}(S)$ is {\em meet complete}.
The map $\iota \colon \: S \rightarrow \mathcal{K}(S)$ is universal for maps to meet complete inverse semigroups.
Thus the inverse semigroup $\mathcal{K}(S)$ is the {\em meet completion} of the inverse semigroup $S$ \cite{Leech}.
It is worth noting that the category of meet complete inverse semigroups and their morphisms is not a full
subcategory of the category of inverse semigroups and their homomorphisms and
so the meet completion of $\mathcal{K}(S)$ is $\mathcal{K}(\mathcal{K}(S))$ and not just $\mathcal{K}(S)$.

At this point, we want to highlight a class of transitive actions that will play an important role both here and in Section~4.

\begin{remark}
{\em Let $T$ be an inverse semigroup and let $e$ be any idempotent in $T$.
We denote by $L_{e}$ the $\mathcal{L}$-class containing $e$.
The set $L_{e}$ therefore consists of all elements $t \in T$ such that $\dom (t) = e$.
Define a partial function from $T \times L_{e}$ to $L_{e}$ by $\exists a \cdot x$ iff $\dom (ax) = e$.
This defines a transitive action of $T$ on $L_{e}$ called the {\em (left) Sch\"utzenberger action determined by the idempotent $e$}.
This is the transitive action determined by the closed inverse subsemigroup $e^{\uparrow}$.}
\end{remark}

The structure of $\mathcal{K}(S)$ is inextricably linked to the structure of transitive actions of $S$.
The following was first stated in \cite{Law1}.

\begin{proposition} Let $S$ be an inverse semigroup.
Every transitive representation of $S$ is the restriction of a Sch\"utzenberger representation of $\mathcal{K}(S)$.
\end{proposition}
\proof Let $H$ be a closed inverse subsemigroup of $S$.
In the inverse semigroup $\mathcal{K}(S)$, the $\mathcal{L}$-class $L_{H}$ of the idempotent $H$
consists of all $A \in \mathcal{K}(S)$ such that $A^{-1} \otimes A = H$.
Let $a \in A$.
Then $A = (aH)^{\uparrow}$.
It follows that $L_{H}$ consists of precisely the left cosets of $H$ in $S$.
Let $A \in L_{H}$ and consider the product $s^{\uparrow} \otimes A$.
Then this again belongs to $L_{H}$ precisely when $(sa)^{-1}sa \in H$ and is equal to $(saH)^{\uparrow}$.
It follows that via the map $\iota$ the inverse semigroup acts on $L_{H}$ precisely as it acts on $S/H$.
\qed \\

If $H$ and $K$ are two idempotents of $\mathcal{K}(S)$ then they are $\mathcal{D}$-related iff there exists $A \in \mathcal{K}(S)$
such that $A^{-1} \otimes A = H$ and $A \otimes A^{-1} = K$ iff $H$ and $K$ are conjugate.
Thus the $\mathcal{D}$-classes of $\mathcal{K}(S)$ are in bijective correspondence with the conjugacy classes of closed inverse subsemigroups.

We may, in some sense, `globalize' the connection between $\mathcal{K}(S)$ and transitive actions of $S$.
Denote by  $\mathbf{O}(S)$ the category whose  objects are the {\em right} $S$-spaces $H/S$
and whose arrows are the (right) morphisms.
We now recall the following construction \cite{Law3'}.
Let $S$ be an inverse semigroup.
We can construct from $S$ a right cancellative category, denoted $\mathbf{R}(S)$,
whose elements are pairs $(s,e) \in S \times E(S)$ such that $\dom (s) \leq e$.
We regard $(s,e)$ as an arrow from $e$ to $\ran (s)$ and define a product by $(s,e)(t,f) = (st,e)$.

The following generalizes Example~2.2.3 of \cite{Law3'}.

\begin{proposition} The category $\mathbf{O}(S)$ is isomorphic to the category $\mathbf{R}(\mathcal{K}(S))$.
\end{proposition}
\proof We observe first that a morphism with a transitive space as its domain is determined by its value on any element of that domain.
Let $\phi \colon \: U/S \rightarrow V/S$ be a morphism.
Then $\phi$ is determined by the value taken by $\phi (U) = (Va)^{\uparrow}$.
Now the stabilizer $S_{U}$ of $U$ is $U$ itself and the stabilizer $S_{(Va)^{\uparrow}}$ is $(a^{-1}Va)^{\uparrow}$.
Thus by Theorem~2.9, we have that $U \subseteq (a^{-1}Va)^{\uparrow}$.
Conversely, if we are given that $U \subseteq (a^{-1}Va)^{\uparrow}$
then we can define a morphism from $U/S$ to $V/S$ by $U \mapsto (Va)^{\uparrow}$.
There is therefore a bijection between morphisms from $U/S$ to $V/S$ and inclusions $U \subseteq (a^{-1}Va)^{\uparrow}$.
We shall encode the morphism $\phi$ by the triple $(V,(Va)^{\uparrow},U)$.
Let $\psi \colon \: V/S \rightarrow W/S$ be a morphism encoded by the triple $(W,(Wb)^{\uparrow},V)$.
The triple encoding $\psi \phi$ is of the form $(W,(Wc)^{\uparrow},U)$ where $\psi \phi (U) = (Wc)^{\uparrow}$.
Thus $(W,(Wb)^{\uparrow},V)(V,(Va)^{\uparrow},U) = (W,(Wba)^{\uparrow},U)$.
The product $(Wb)^{\uparrow} \otimes (Va)^{\uparrow}$ in $\mathcal{K}(S)$ is precisely $(Wba)^{\uparrow}$.
We now recall that the natural partial order in $\mathcal{K}(S)$ is reverse inclusion.
It follows that the triple $(V,(Va)^{\uparrow},U)$ can be identified with the pair
$((Va)^{\uparrow},U)$ where $\dom ((Va)^{\uparrow}) \leq U$.
We regard $((Va)^{\uparrow},U)$ as an arrow with domain $U$ and codomain $V$.
The result now follows.\qed \\

\subsection{The inverse semigroup of filters $\mathcal{L}(S)$}

We shall now describe an inverse subsemigroup of $\mathcal{K}(S)$.
A subset $A \subseteq S$ of an inverse semigroup $S$ is said to be {\em (down) directed} if it is non-empty and, for each $a,b \in A$,
there exists $c \in A$ such that $c \leq a, b$.
Closed directed sets in a poset are called \emph{filters}.
When this definition is applied to semilattices then we recover the definition given earlier.

\begin{lemma} The closed directed subsets are precisely the directed cosets.
\end{lemma}
\proof A directed coset is certainly a closed directed subset.
Let $A$ be a closed directed subset.
We prove that it is an atlas.
Clearly $A \subseteq AA^{-1}A$.
Thus we need only check that $AA^{-1}A \subseteq A$.
Let $a,b,c \in A$.
Then since $A$ is directed there is $d \in A$ such that $d \leq a,b,c$.
Thus $d = dd^{-1}d  \leq ab^{-1}c$ and so $ab^{-1}c \in A$ since $A$ is also closed.\qed \\

\begin{lemma} A closed inverse subsemigroup $T$ of an inverse semigroup $S$ is directed if and only if
there is a filter $F \subseteq E(S)$ such that $T = F^{\uparrow}$.
\end{lemma}
\proof Suppose that $T = F^{\uparrow}$.
Let $a,b \in T$.
Then $e \leq a$ and $f \leq b$ for some $e,f \in F$.
But $F$ is a filter in the semilattice of idempotents and so closed under multiplication.
Thus $ef \in F$.
But then $ef \leq a,b$ and so $T$ is directed.

Let $T$ be a closed directed inverse subsemigroup.
Put $F = E(S)$.
Let $e,f \in F$.
Now $T$ is directed and so there is $i \in T$ such that $i \leq e,f$.
Thus $i$ is an idempotent.
But $i \leq ef \leq e,f$ and so, since $F$ is closed, we have that $ef \in F$.
It follows that $F$ is a filter in $E(S)$.
Clearly $F^{\uparrow} \subseteq T$.
Let $t \in T$.
Then $t^{-1}t \in T$ since $T$ is an inverse subsemigroup.
But $T$ is directed so there exists $j \leq t,t^{-1}t$.
But then $j$ is an idempotent and  so $j \leq t$ gives that $t \in F^{\uparrow}$.
Hence $T \subseteq F^{\uparrow}$.
Thus $T = F^{\uparrow}$, as required.\qed \\

\begin{lemma} If $A$ and $B$ are both directed cosets then $(AB)^{\uparrow}$ is the smallest directed coset containing $AB$;
it is also the smallest coset containing $AB$.
\end{lemma}
\proof The set $(AB)^{\uparrow}$ is closed so we need only show it is directed.
Let $ab,a'b' \in AB$.
Then there exists $c \leq a,a'$ where $c \in A$ and $d \leq b,b'$ where $d \in B$.
It follows that $cd \in AB$ and $cd \leq ab,a'b'$.
Thus the set is directed.

Now let $X$ be any coset containing $AB$.
Then $X$ is closed and so $(AB)^{\uparrow} \subseteq X$.\qed \\

The subset of $\mathcal{K}(S)$ consisting of directed cosets is denoted by $\mathcal{L}(S)$.

\begin{proposition} Let $S$ be an inverse semigroup.
\begin{description}
\item[{\rm (i)}] $\mathcal{L}(S)$ is an inverse subsemigroup of $\mathcal{K}(S)$.
\item[{\rm (ii)}] The directed cosets of $S$ are precisely the cosets of the closed directed inverse subsemigroups of $S$.
\item[{\rm (iii)}] Each element of $\mathcal{K}(S)$ is the meet of a subset of $\mathcal{L}(S)$ contained in an $\mathcal{H}$-class of $\mathcal{L}(S)$.
\end{description}
\end{proposition}
\proof (i) If $A,B \in \mathcal{K}(S)$ then their product is the intersection of all cosets containing $AB$.
But if $A,B \in \mathcal{L}(S)$ then by Lemma~3.6 this intersection will also belong to $\mathcal{L}(S)$.
Closure under inverses is immediate.
Thus $\mathcal{L}(S)$ is an inverse subsemigroup of $\mathcal{K}(S)$.

(ii) If $A \in \mathcal{K}(S)$ then $A = (aH)^{\uparrow} = (a)^{\uparrow} \otimes H$ where $H = A^{-1} \otimes A$ and $a \in A$.
Thus $A$ is directed if and only if $H$ is directed.

(iii) Let $A \in \mathcal{K}(S)$ be a coset. Define a relation $\sim$ on the set $A$ by $a \sim b$ iff there exists $c \in A$
such that $c \leq a,b$. We show that $\sim$ is an equivalence relation on $A$.
Clearly $\sim$ is reflexive and symmetric.
It only remains to prove that it is transitive.
Let $a \sim b$ and $b \sim c$.
Then there exists $x \leq a,b$ and $y \leq b, c$ where $x,y \in A$.
In particular, $x,y \leq b$.
Thus $z = xy^{-1}y = yx^{-1}x$ is the meet of $x$ and $y$.
Since $A$ is a coset $xy^{-1}y, yx^{-1}x \in A$.
It follows that $z \leq a,c$.
Denote the blocks of the partition induced by $\sim$ on $A$ by $A_{i}$ where $i \in I$.
Each block is directed by construction and easily seen to be closed.
It follows that each block is a directed coset and so $A_{i} \in \mathcal{L}(S)$.
We have therefore proved that $A = \bigwedge_{i \in I} A_{i}$.

It remains to show that $A_{i} \,\mathcal{H} \, A_{j}$.
To do this it is enough to compute $A_{i}^{-1} \otimes A_{i}$ and $A_{i} \otimes A_{i}^{-1}$
and observe that these idempotents do not depend on the suffix $i$.
We may write $A = (aH)^{\uparrow}$ for some closed inverse subsemigroup $H$ of $S$ and element $a$ such that $\mathbf{d}(a) \in H$.
Put $F = E(H)$ the semilattice of idempotents of $H$.
Put $K = F^{\uparrow}$ and $L = (aKa^{-1})^{\uparrow}$,
both closed directed inverse subsemigroups of $S$ and so elements of $\mathcal{L}(S)$.
We prove that $K = A_{i}^{-1} \otimes A_{i}$ and $L = A_{i} \otimes A_{i}^{-1}$.
From $A \leq A_{i}$ we have that $H = A^{-1} \otimes A \leq A_{i}^{-1} \otimes A_{i}$
and $(aHa^{-1})^{\uparrow} \leq A_{i} \otimes A_{i}^{-1}$.
By construction $H \leq K$ and $K$ is in fact the smallest idempotent of $\mathcal{L}(S)$ above $H$.
It follows that $K \leq  A_{i}^{-1} \otimes A_{i}$ and similarly $L \leq A_{i} \otimes A_{i}^{-1}$.
It remains to show that equality holds in each case which means checking that
$K \subseteq A_{i}^{-1} \otimes A_{i}$ and $L \subseteq A_{i} \otimes A_{i}^{-1}$.

Let $k \in K$ and $a_{i} \in A_{i}$.
Now $k \in K \subseteq H$ and $a_{i} \in A_{i} \subseteq A$.
Thus $a_{i}k \in A$.
But $a_{i}k \leq a_{i}$.
Now if $a_{i}k \in A_{j}$ then by closure $a_{i} \in A_{j}$ and so we must have that $a_{i}k \in A_{i}$.
Thus $ka_{i}^{-1}a_{i} \in A_{i}^{-1} \otimes A_{i}$ and so by closure $k \in A_{i}^{-1} \otimes A_{i}$, as required.

Let $l \in L$.
Let $a_{i} \in A_{i}$.
Then $A = (a_{i}H)^{\uparrow}$.
Thus $L = (a_{i}Ka_{i}^{-1})^{\uparrow}$.
It follows that $a_{i}^{-1}la_{i} \in K$
and so $a_{i}a_{i}^{-1}l \in a_{i}Ka_{i}^{-1}$
giving $l \in A_{i} \otimes A_{i}^{-1}$.

An alternative way of proving this result is to observe that $K$ is a closed inverse subsemigroup of $H$
and so $H$ can be written as a disjoint union of some of the left cosets of $K$.
We can then use this decomposition to write $A$ itself as a disjoint union of left cosets of $K$.\qed \\

We say that an inverse semigroup $S$ is {\em meet complete} if every non-empty subset of $S$ has a meet.
Meet completions of inverse semigroups are discussed at the end of Section~1.4 of \cite{Law2},
\cite{Law1} and most importantly in \cite{Leech}.
The meet completion of an inverse semigroup $S$ is in fact $\mathcal{K}(S)$ \cite{Leech}.

The inverse semigroup $S$ is said to have all {\em directed meets} if it has meets of all non-empty directed subsets.
The result below shows that $\mathcal{L}(S)$ is the {\em directed} meet completion of $S$
in the same way that $\mathcal{K}(S)$ is the meet completion.

\begin{proposition} Let $S$ be an inverse semigroup.
Then $\mathcal{L}(S)$ is the directed meet completion of $S$.
\end{proposition}
\proof We have the embedding $\iota \colon \: S \rightarrow \mathcal{L}(S)$
and once again each $A \in \mathcal{L}(S)$ is the join of all the $s^{\uparrow}$ where $s \in A$.
This time the set over which we are calculating the meet is directed.
Let $\mathcal{A} = \{A_{i} \colon \: i \in I \}$ be a directed subset of $K(S)$.
Thus for each pair of cosets $A_{i}$ and $A_{j}$ there is a coset $A_{k}$ such that $A_{k} \leq A_{i}, A_{j}$.
Put $A = \bigcup_{i \in I} A_{i}$.
It is clearly a closed subset.
If $a,b \in A$ then $a \in A_{i}$ and $b \in A_{j}$ for some $i$ and $j$.
By assumption $A_{i},A_{j} \subseteq A_{k}$ for some $k$.
Thus $a,b \in A_{k}$.
But $A_{k}$ is a directed subset and so there exists $c \in A_{k}$ such that $c \leq a,b$.
It follows that $A$ is a closed and directed subset and so is a directed coset by Lemma~3.4.
 It is now immediate that $A$ is the meet of the set $\mathcal{A}$.
Let $\theta \colon \: S \rightarrow T$ be a homomorphism to an inverse semigroup $T$ which has all meets of directed subsets.
Define $\psi \colon \: \mathcal{K}(S) \rightarrow T$ by $\psi (A) = \bigwedge \theta (A)$.
Then $\psi$ is a homomorphism and the unique one such that $\psi \iota = \theta$.\qed \\

In \cite{Lenz}, Lenz constructs an inverse semigroup $\mathcal{O}(S)$ from an inverse semigroup $S$,
which is the basis for his \'etale groupoid associated with $S$.
The key result for our paper is the following.

\begin{theorem} The inverse semigroup $\mathcal{L}(S)$ is isomorphic to Lenz's semigroup $\mathcal{O}(S)$.
\end{theorem}
\proof Let $\mathcal{F} = \mathcal{F}(S)$ denote the set of directed {\em subsets} of $S$.
For $A,B \in \mathcal{F}$ define $A \prec B$ iff for each $b \in B$ there exists $a \in A$ such that $a \leq b$.
This is a preorder.
The associated equivalence relation is given by $A \sim B$ iff $A \prec B$ and $B \prec A$.
We now make two key observations.
(1) $A \sim A^{\uparrow}$. It is easy to check that $A^{\uparrow}$ is directed.
By definition $A \prec A^{\uparrow}$, whereas $A^{\uparrow} \prec A$ is immediate.
(2) $A^{\uparrow} \sim B^{\uparrow}$ iff $A^{\uparrow} = B^{\uparrow}$.
There is only one direction needs proving.
Suppose that $A^{\uparrow} \sim B^{\uparrow}$.
Let $a \in A^{\uparrow}$.
Then $B^{\uparrow} \prec A^{\uparrow}$ and so there is $b \in B$ such that $b \leq a$.
But then $a \in B^{\uparrow}$.
Thus $A^{\uparrow} \subseteq B^{\uparrow}$.
The reverse inclusion is proved similarly.
By (1) and (2), it follows that $A \sim B$ iff $A^{\uparrow} = B^{\uparrow}$.
As a set, $\mathcal{O}(S) = \mathcal{F}(S)/\sim$.
We have therefore set up a bijection between $\mathcal{O}(S)$ and $\mathcal{L}(S)$.
Lemma~3.6 tells us that the multiplication defined in \cite{Lenz} in $\mathcal{O}(S)$
ensures that this bijection is an isomorphism.\qed \\

Denote by  $\mathbf{U}(S)$ the category whose  objects are the {\em right} $S$-spaces $H/S$
where $H$ is directed and whose arrows are the (right) morphisms.
We have the following analogue of Proposition~3.3.

\begin{proposition}
The category $\mathbf{U}(S)$ is isomorphic to the category $\mathbf{R}(\mathcal{L}(S))$.
\end{proposition}

\subsection{Paterson's \'etale groupoid}

Theorem~3.9 brings us to the beginning of Section~4 of Lenz's paper \cite{Lenz} where he describes Paterson's \'etale groupoid.
If $T$ is an inverse semigroup, then it becomes a groupoid when we define a partial binary operation $\cdot$, called the {\em restricted product},
by $\exists s \cdot t$ if and only if $\dom (s) = \ran (t)$ in which case $s \cdot t = st$.
Paterson's groupoid is precisely $(\mathcal{L}(S),\cdot)$ equipped with a suitable topology.
The isomorphism functor defined by Lenz from $\mathcal{L}(S)$ to Paterson's groupoid can be very easily described in terms of the ideas introduced in our paper.
Let $A \in \mathcal{L}(S)$.
Define $P = (AA^{-1})^{\uparrow}$.
Then for any $a \in A$ we have that $A = (Pa)^{\uparrow}$.
Thus we may regard $A$ as a {\em right coset} of the closed, directed inverse subsemigroup $P$.
By the dual of Lemma~2.5(i), we have that $(Pa)^{\uparrow} = (Pb)^{\uparrow}$,
where $aa^{-1},bb^{-1} \in P$, if and only if $ab^{-1} \in P$ if and only if $pa = pb$ for some $p \in P$,
where we use the fact that every element of $P$ is above an idempotent.
The ordered pair $(P,a)$ where $\ran (a) \in P$ determines the right coset $(Pa)^{\uparrow}$
and another such pair $(P,b)$ determines the same right coset if and only if $pa = pb$ for some $p \in P$.
This leads to an equivalence relation and we denote the equivalence class containing  $(P,a)$ by $[P,a]$.
The isomorphism functor between the Lenz groupoid $\mathcal{L}(S)$ and Paterson's groupoid is therefore defined by
$A \mapsto [(AA^{-1})^{\uparrow},a]$ where $a \in A$.
We see that Paterson has to work with equivalence classes because of the non-uniqueness of coset-respresentatives,
and Lenz has to work with equivalence classes because he works with generating sets of filters rather than with the filters themselves.
In our approach, the use of equivalence classes in both cases is avoided.

Recall from Section~2.2, that a transitive $S$-space $X$ is universal if the stabilizer $H$ of a point of $X$ is $F^{\uparrow}$ where $F$ is a filter in $E(S)$.
In other words, by Lemma~3.5 the closed inverse subsemigroup $H$ is directed.
It follows that the universal transitive actions of $S$ are determined by the directed filters that are also inverse subsemigroups.
We shall now describe how the structure of the groupoid $(\mathcal{L}(S),\cdot)$ reflects the properties of transitive actions of $S$.
In what follows, we can just as easily work in the inverse semigroup as in the groupoid.

\begin{proposition} Let $S$ be an inverse semigroup.
\begin{enumerate}

\item The connected components of the groupoid $\mathcal{L}(S)$ are in bijective correspondence with the equivalence
classes of universal transitive actions of $S$.

\item Let $H$ be an identity in $\mathcal{L}(S)$.
Then the local group $G_{H}$ at $H$ is isomorphic to the group $\overline{E(H)}/\sigma$.

\end{enumerate}
\end{proposition}
\proof
(1) The identities of $\mathcal{L}(S)$ are the closed directed inverse subsemigroups of $S$.
Two such identities belong to the same connected component if and only if they are conjugate.
The result now follows by Proposition~2.7.

(2) Put $F = E(H)$ so that $H = \overline{F}$.
Let $A$ be in the local group determined by $H$.
Then $H = (A^{-1}A)^{\uparrow} = (AA^{-1})^{\uparrow}$.
Define $\theta \colon G_{H} \rightarrow \overline{E(H)}/\sigma$ by $\theta (A) = \sigma (a)$ where $a \in A$.

We show first that this map is well-defined.
Let $f \in F$ and let $a \in A$.
Then $a^{-1}fa \in A^{-1}E(A)A \subseteq HA = A$ and so $a^{-1}fa \in F$
and
$afa^{-1} \in AE(A)A^{-1} \subseteq AH = A$ and so $afa^{-1} \in F$.
Thus $A \subseteq \overline{F}$.
Next suppose that $a,b \in A$.
Then there is an element $c \in A$ such that $c \leq a,b$.
Thus $\sigma (a) = \sigma (b)$.
It follows that $\theta$ is well-defined.

We now show that $\theta$ defines a bijection.
Suppose that $\theta (A) = \theta (B)$.
Then $a \sigma b$ where $a \in A$ and $b \in B$.
Thus there exists $c \in \overline{F}$ such that $c \leq a,b$.
It follows that $c = ac^{-1}c = bc^{-1}c$ and so $a^{-1}ac^{-1}c \leq a^{-1}b$.
But  $a^{-1}ac^{-1}c \in F$ and so $A = B$.
Thus $\theta$ is injective.
Let $a \in \overline{F}$.
Then $a^{-1}a, \in \overline{F}$ and so $a^{-1}a \in F$.
Thus $A = (aH)^{\uparrow}$ is a well-defined coset and
Then $(A^{-1}A)^{\uparrow} = H = (AA^{-1})^{\uparrow}$.
It follows that $A \in G_{H}$ and $\theta (A) = \sigma (a)$.
Thus $\theta$ is surjective.

Finally we show that $\theta$ defines a homomorphism.
Let $A,B \in G_{H}$ and $a \in A$ and $b \in B$.
By Lemma~3.6, $A \otimes B = (AB)^{\uparrow}$ and contains $ab$.
Thus
$\theta (A)\theta (B) = \sigma (a)\sigma(b) = \sigma (ab) = \theta (A \otimes B)$.
\qed \\

We now have the following theorem.

\begin{theorem} Let $S$ be an inverse semigroup.
Then $\mathcal{L}(S)$ explicitly encodes universal transitive actions of $S$ via its Sch\"utzenberger actions,
and implicitly encodes all transitive actions via its local groups.
\end{theorem}
\proof An idempotent of $\mathcal{L}(S)$ is just an inverse subsemigroup $H$ of $S$ that is also a filter.
Denote by $\mathcal{L}_{H}$ the $\mathcal{L}$-class of $H$ in the inverse semigroup $\mathcal{L}(S)$.
The elements of $\mathcal{L}_{H}$ are just the left cosets of $H$ in $S$.
The inverse semigroup $\mathcal{L}(S)$ acts on the set $\mathcal{L}_{H}$, a Sch\"utzenberger action,
and so too does $S$ via the map $\iota$ of Proposition~3.8.
This latter action is equivalent to the action of $S$ on $S/H$.
We have therefore shown that $\mathcal{L}(S)$ encodes universal transitive actions of $S$ via its Sch\"utzenberger actions.

By Proposition~2.15(1) each transitive action of $S$ on a set $Y$ is strongly covered by a universal one $X$.
Let $H$ be a stabilizer of this universal action of $S$ on $X$.
Then the strong covering is determined by a strong congruence which by Theorem~2.16 is determined
by a subgroup of the $\mathcal{H}$-class in $\mathcal{L}(S)$ containing the idempotent $H$;
in other words, by a subgroup of the local group determined by the idempotent $H$.\qed \\

Finally, the topology on the groupoid $\mathcal{L}(S)$ is defined in terms of the embedding $S \rightarrow \mathcal{L}(S)$ as follows.
Let $s \in S$.
Define
$$U_{s} = \{A \in \mathcal{L}(S) \colon s \in A \}$$
and for $s_{1}, \ldots, s_{n} \leq s$ define
$$U_{s;s_{1}, \ldots, s_{n}} = U_{s} \cap U_{s_{1}}^{c} \cap \ldots \cap U_{s_{n}}^{c}.$$
Then the sets $U_{s;s_{1}, \ldots, s_{n}}$ form a basis for a topology.

\section{Matrix representations of inverse semigroups}

We deduce here results of the third author on the finite dimensional irreducible representations of inverse semigroups~\cite{S}.
There an approach based on groupoid algebras was used,
whereas here we use results of J.~A.~Green~\cite[Chapter 6]{Green}
and the universal property of $\mathcal{L}(S)$.

\subsection{Green's theorem and primitive idempotents}

The following theorem summarizes the contents of~\cite[Chapter 6]{Green}.  Let $A$ be a ring.
A module is assumed to be a left $A$-module unless otherwise stated.
We also consider only \emph{unitary} $A$-modules, that is, $A$-modules $M$ such that $AM=M$ (where $AM$ means the submodule generated by elements $am$ with $a\in A$ and $m\in M$).  If $A$ has a unit, then this is the same as saying that the unit acts as the identity on $M$.  In particular, a \emph{simple} $A$-module is an $A$-module $M$ such $AM\neq 0$ and there are no non-zero proper submodules of $M$. If $e$ is an idempotent of $A$ and $M$ is an $A$-module, then $eM$ is an $eAe$-module.
The functor $M\mapsto eM$ is called {\em restriction} and we sometimes denote it $\Res_e(M)$. It is well known and easy to check that $eM\cong \Hom_A(Ae,M)$, where the latter has a left $eAe$-action induced by the right action of $eAe$ on $Ae$.  For an $eAe$-module $N$, define \[\Ind_e(N)=Ae\otimes_{eAe} N.\]  The usual hom-tensor adjunction implies that $\Ind_e$ is the left adjoint of $\Res_e$. Moreover, $\Res_e\Ind_e$ is isomorphic to the identity functor on the category $eAe$-modules.  Indeed, $eae\otimes n\mapsto eaen$ is an isomorphism with inverse $n\mapsto e\otimes n$.  These isomorphisms are natural in $N$.

\begin{theorem}[Green]\label{Greenthm}
Let $A$ be a ring and $e\in A$ an idempotent.
\begin{enumerate}
\item If $N$ is a simple $eAe$-module, then the induced module \[\Ind_e(N)=Ae\otimes_{eAe} N\] has a unique maximal submodule $R(N)$, which can be described as the largest submodule of $\Ind_e(N)$ annihilated by $e$. Moreover, the simple module $\widetilde N=\Ind_e(N)/R(N)$ satisfies $N\cong e\widetilde N$.
\item If $M$ is a simple $A$-module with $eM\neq 0$, then $eM$ is a simple $eAe$-module and $M\cong \widetilde {eM}$.
\end{enumerate}
\end{theorem}

Let $S$ be an inverse semigroup and suppose that $e$ is a minimum idempotent of $S$.
Then $eSe=G_e$, the maximal subgroup of $S$ at $e$, and is also the maximal group image of $S$.  Moreover, $Se=G_e=eS$ and the action of $S$ on the left of $Se$ factors through the maximal group image homomorphism.  Let $k$ be a commutative ring with unit.  Then $ekSe\cong kG_e$ and so Green's theorem shows that simple $kS$-modules $M$ with $eM\neq 0$ are in bijection with simple $kG_e$-modules via induction and restriction.  Moreover, since $kSe=kG_e$, we have that $\Ind_e(N)=N$ with the action of $S$ induced by the maximal group image homomorphism.  Thus $\Ind_e(N)$ already is a simple $kS$-module.  Let us consider the analogous situation for primitive idempotents.

Let $e$ be a primitive idempotent of an inverse semigroup with $0$.  Observe that in this case $eSe=G_e\cup \{0\}$ since $e,0$ are the only idempotents of $eSe$ and so if $s\neq 0$, then $ss^{-1}=e=s^{-1} s$.  Thus if $k_0S$ is the contracted semigroup algebra of $S$ (meaning the quotient of $kS$ by the ideal of scalar multiples of the zero of $S$), then $ek_0Se\cong kG_e$ and so again by Green's theorem, we have a bijection between simple $k_0S$-modules $M$ with $eM\neq 0$ and $kG_e$-modules via induction.  We aim to show now that if $N$ is a simple $kG_e$-module, then $\Ind_e(N)$ is already a simple $k_0S$-module.
Let $L_e$ be the $\mathcal L$-class of $e$.
Then since $e$ is primitive, it follows that $L_e = Se\setminus \{0\}$ and so $k_0Se=kL_e$ where $S$ acts on the left of $kL_e$ via linearly extending the left Sch\"utzenberger representation.
The group $G_e$ acts freely on the right of $L_e$ with orbits the $\mathcal H$-classes contained in $L_e$.  Thus $k_0Se=kL_e$ is free as a right $ek_0Se=kG_e$-module.  Let $T$ be a transversal to the $\mathcal{H}$-classes of $L_e$ and let $N$ be a $kG_e$-module.  Then as a $k$-module, $\Ind_e(N) = \bigoplus_{t\in T} t\otimes_k N$.  A fact we shall use is that any element of $L_e$ is primitive and so if $t_1\neq t_2\in T$, then $t_1t_1^{-1}\neq t_2t_2^{-1}$ and hence $t_1t_1^{-1} t_2=0$.

\begin{lemma}\label{inducedaresimplelemma}
If $N$ is a non-zero $kG_e$-module, then no non-trivial submodule of $\Ind_e(N)$ is annihilated by $e$.
\end{lemma}
\begin{proof}
Let $M$ be a non-zero submodule of $\Ind_e(N)$.  Notice that $M$ is annihilated by $e$ if and only if it is annihilated by the  ideal generated by $e$.  So let $m=\sum_{t\in T} t\otimes n_t$ (with only finitely many terms non-zero) be a non-zero element of $M$.  Then there exists $t\in T$ with $n_t\neq 0$.  By the observation just before the proof $tt^{-1} m = t\otimes n_t\neq 0$ and so $tt^{-1}$ does not annihilate $m$.  But $e=t^{-1} t$ generates the same ideal as $tt^{-1}$ and so $M$ is not annihilated by $e$.
\end{proof}

As a corollary, we obtain from Green's Theorem~\ref{Greenthm} that if $N$ is a simple $kG_e$-module, then $\Ind_e(N)$ is a simple $k_0S$-module.

\begin{corollary}\label{simple}
Let $S$ be an inverse semigroup, $e\in E(S)$ a primitive idempotent and $k$ a commutative ring with unit. If $N$ is a simple $kG_e$-module, then $\Ind_e(N)$ is a simple $kS$-module.
\end{corollary}

If $k$ is a field,
then from $\Ind_e(N) = \bigoplus_{t\in T} t\otimes_k N$,
we see that $\Ind_e(N)$ is finite dimensional if and only if $T$ is finite and $N$ is finite dimensional.

\subsection{The main result}

Suppose now that $S$ is any inverse semigroup and $e\in E(S)$.
Let $I_e=SeS\setminus J_e$ be the ideal of elements strictly $\mathcal J$-below $e$.
If $N$ is a $kG_e$-module, then let \[\Ind_e(N) = k_0[S/I_e]e\otimes_{kG_e} N=(kS/kI_e)e\otimes _{kG_e} N.\]
Equivalently, if $L_e$ is the $\mathcal{L}$-class of $e$, then $kL_e$ is a free right $kG_e$-module with basis the set of $\mathcal H$-classes of $L_e$
and also it is a left $kS$-module by means of the action of $S$ on the left of $L_e$ by partial bijections
via the Sch\"utzenberger representation.  Then $\Ind_e(N) = kL_e\otimes _{kG_e} N$.  Suppose now that the $\mathcal D$-class of $e$ contains only finitely many idempotents; in this case we say that $e$ has \emph{finite index} in $S$.  Under the hypothesis that $e$ has finite index it is well known that if $f\in E(S)$ with $f<e$, then $SfS\neq SeS$ and so $f\in I_e$.  Thus $e$ is primitive in $S/I_e$ and so Corollary~\ref{simple} shows that $\Ind_e(N)$ is simple for any simple $kG_e$-module in this setting.

We are now ready to construct the finite dimensional irreducible representations of an inverse semigroup over a field.
This was first carried out by Munn~\cite{Munn}, whereas the construction presented here first appeared in~\cite{S}
where it was deduced as a special case of a result on \'etale groupoids.
Our approach here uses the inverse semigroup $\mathcal{L}(S)$.
Fix a field $k$.
First we construct a collection of simple $kS$-modules.

\begin{proposition}\label{constructirreps}
Let $e\in E(\mathcal{L}(S))$ have finite index and let $N$ be a simple $kG_e$-module.
Then $\Ind_e(N)$ is a simple $kS$-module.  Moreover, $\Ind_e(N)$ is finite dimensional if and only if $N$ is.
\end{proposition}
\begin{proof}
The above discussion shows that $\Ind_e(N)$ is simple as a $k\mathcal{L}(S)$-module so we just need to show that any $S$-invariant subspace is $\mathcal{L}(S)$-invariant.   In fact, we show that each element of $\mathcal{L}(S)$ acts the same on $\Ind_e(B)$ as some element of $S$.  It will then follow that any $S$-invariant subspace is $\mathcal{L}(S)$-invariant and so $\Ind_e(N)$ is a simple $kS$-module.

Let $T$ be a transversal for the orbits of $G_e$ on $L_e$.  Then $T$ is finite since these orbits are in bijection with $\mathcal R$-classes of $D_e$, which in turn are in bijection with idempotents of $D_e$.
Let $A\in \mathcal {L}(S)$ and write $A=\bigwedge_{d\in D}s_d$ with $s\in S$ and $D$ a directed set.  We claim that if $t\otimes n$ is an elementary tensor with $t\in T$, then there exists $d_t\in D$ depending only on $t$ (and not $n$) such that $A(t\otimes n)=s_d(t\otimes n)$ for all $d\geq d_t$. By~\cite[Section 1.4, Proposition 19]{Law2}, we have $At = \bigwedge_{d\in D}(s_dt)$.
Since the $\mathcal D$-class of $e$ has only finitely many idempotents,
it follows by Theorem~3.2.16 of \cite{Law2} that distinct elements of $\mathcal D$ are not comparable in the natural partial order.
Since the set $\{s_dt\mid d\in D\}$ is directed, either $s_dt\lneq_{\mathcal L}e$ for all sufficiently large elements of $D$ or $s_dt$ is an element $\ell$ of $L_e$ independent of $d$.  In the first case $At\lneq_{\mathcal L} e$ and in the second case $At=\ell$.  Thus in the first case, $A(t\otimes n) = 0 = s_d(t\otimes n)$ for $d$ large enough, whereas in the second case $A(t\otimes n)=\ell\otimes n=s_d(t\otimes n)$ for all $d\in D$.  We conclude $d_t$ exists.

Since $T$ is finite, we can find $d_0\in D$ with $d_0\geq d_t$ for all $t\in R$.  Then $A$ and $s_{d_0}$ agree on all elements of the form $t\otimes n$ with $t\in T$ and $n\in N$.  But such elements span $\Ind_e(N)$ and so we conclude that $A$ and $s_{d_0}$ agree on $\Ind_e(N)$.

The final statement follows from the previous discussion.
\end{proof}

Note that application of the restriction functor and the fact that $\Res_e\Ind_e$ is isomorphic to the identity shows that $\Ind_e(N)\cong \Ind_e(M)$ implies $N\cong M$.
Also, if $e,f$ are two finite index idempotents of $\mathcal L(S)$ and $e\nleq_{\mathcal J} f$, then $f$ annihilates $\Ind_e(N)$ for any $kG_e$-module and hence all elements of $f$, viewed as a filter, annihilate $\Ind_e(N)$. On the other hand, no element of the filter $e$ annihilates $\Ind_e(N)$.  It follows that if $e,f$ are finite index idempotents that are not $\mathcal{D}$-equivalent, then the modules of the form $\Ind_e(N)$ and $\Ind_f(M)$ are never isomorphic.  Clearly, $\mathcal{D}$-equivalent idempotents give isomorphic collections of simple modules.  Thus, for each $\mathcal{D}$-class with finitely many idempotents, we get a distinct set of simple $kS$-modules (up to isomorphism).

The following fact is well known and easy to prove.

\begin{proposition}\label{finitelymanyidempotents}
Let $k$ be a field and $V$ an $n$-dimensional $k$-vector space.  Then any semilattice  in $\mathrm{End}_k(V)$ has size at most $2^n$.
\end{proposition}
\begin{proof}
Any idempotent matrix is diagonalizable and so any semilattice of matrices is simultaneously diagonalizable. But the multiplicative monoid of $k^n$ has $2^n$ idempotents.
\end{proof}

We can now complete the description of the finite dimensional irreducible representations of an inverse semigroup.
In the statement of the theorem below,
it is worth recalling that
$e = H$ is a finite index, closed directed subsemigroup of $S$
and $G_{e}$ is the group $\overline{E(H)}/\sigma$ described in Theorem~2.16.

\begin{theorem}
Let $k$ be a field and $S$ an inverse semigroup.
Then the finite dimensional simple $kS$-modules are precisely those of the form $\Ind_e(N)$
where $e$ is a finite index idempotent of $\mathcal{L}(S)$
and $N$ is a finite dimensional simple $kG_e$-module.
\end{theorem}
\begin{proof}
It remains to show that every simple $kS$-module $M$ is of this form.  Let $\theta\colon S\to \mathrm{End}_k(V)$ be the corresponding irreducible representation.  Then $T=\theta(S)$ is an inverse semigroup with finitely many idempotents and so trivially directed meet complete.  Thus $\theta$ extends to a homomorphism $\overline{\theta}\colon \mathcal{L}(S)\to \mathrm{End}_k(V)$ by the universal property.  Trivially $\overline{\theta}$ must be irreducible as well.  Let $f$ be a minimal non-zero idempotent of $T=\theta(S)=\overline{\theta}(\mathcal{L}(S))$.  Then $\overline{\theta}^{-1}(f)$ is directed and so has a minimum element $e$.

Suppose $e'\mathcal{D} e$.     Suppose $e''<e'$.  We claim $\overline{\theta}(e'')=0$.  Indeed, choose $A\in \mathcal{L(S)}$ such that $A^{-1}A=e$ and $AA^{-1}=e'$.  Then $A^{-1}e''A< A^{-1}e'A=e$ and so $\overline{\theta}(A^{-1}e''A) =0$.  Thus $\overline{\theta}(e'') = \overline{\theta}(AA^{-1}e''AA^{-1})=0$.  We conclude $\overline{\theta}$ is injective on the idempotents of $D_e$.  Otherwise, we can find $e_1,e_2\in D_e$ with $\overline{\theta}(e_1)=\overline{\theta}(e_2)$.  Then $e_1e_2\leq e_1,e_2$ and $\overline{\theta}(e_1)=\overline{\theta}(e_1e_2)=\overline{\theta}(e_2)$.  Thus $e_1=e_1e_2=e_2$ by the above claim.  We conclude that $e$ has finite index since $T$ has finitely many idempotents.

By choice of $e$, it now follows that $\overline{\theta}$ factors through $S/I_e$ and hence is a $k_0[S/I_e]$-module.  Moreover, $e$ is primitive in $S/I_e$.  (If $I_e=\emptyset$, then we interpret $k_0[S/I_e]$ as $kS$ and $e$ is the minimum idempotent.)    Since $eM=fM\neq 0$ by choice of $f$, it follows by Green's theorem that $N=eM$ is a simple $ek_0[S/I_e]e=kG_e$-module, necessarily finite dimensional.  The identity map $N\to eM$ corresponds under the adjunction to a non-zero homomorphism $\psi\colon \Ind_e(N)\to M$.  But we already know that $\Ind_e(N)$ is simple by Proposition~\ref{constructirreps}.  Schur's lemma then yields that $\psi$ is an isomorphism.  This completes the proof.
\end{proof}

\section*{Appendix}
After an early version of this paper existed we discovered that Jonathon Funk and Pieter Hofstra independently arrived at what we call universal actions, and which they call torsors~\cite{FunkHofstra}.  They show that these correspond exactly to the points of the classifying topos of the inverse semigroup.  Further connections between our work and their work will be explored in an upcoming paper by Funk, Hofstra and the third author.  In particular, we connect the filter construction of Paterson's groupoid with the soberification of the inductive groupoid of the inverse semigroup and the soberification of the inverse semigroup.  We also show that actions of the inverse semigroup on sober spaces correspond to actions of the soberification of the inductive groupoid on sober spaces.


\end{document}